\documentclass[11pt,a4paper,reqno]{amsart}
\usepackage{texmac,enumitem}
\usepackage[foot]{amsaddr}
\usepackage[all]{xy}
\usepackage{amsmath,mathtools}
\xyoption{web}
\usepackage{xifthen}

\usepackage{color}

\operators{Aut,Div,Pic,Pid,Id,Cl,ab,Norm,cts,Ext,ia,opp,ie,J,M,K,G,D}
\operators{skew,com,Core,fin,op,id,continuous,G,Ar,res,cor,Irr,Ann,Maxspec,Nm,cl,na,f}
\calsymbols{c}{O,C,B,A,K,M,D,S,F,X,Y,L,H,P,E,N}
\bbsymbols{b}{E,A,P}
\fraksymbols{f}{F,p,A,M,m,r,c,l,N}

\newcommand{\Ara}[2][]{\ifthenelse{\isempty{#1}}{\widetilde{\Ar}_{#2}}{\widetilde{\Ar}_{#2,#1}}}

\newcommand{\PicTilde}[1]{\widetilde{\Pic}^0_{#1}}

\def\AbsNm{|\!\Nm\!|}
\def\Kbar{\bar{K}}
\DeclareMathOperator{\comp}{c}
\def\centre{Z}
\def\involGrothen{{\rm j}}
\begin{document}

\title[Arakelov class groups of random number fields]
{Arakelov class groups \\ of random number fields}
\author{Alex Bartel$^1$}
\address{$^1$School of Mathematics and Statistics, University of Glasgow,
University Place, Glasgow G12 8SQ, United Kingdom}
\author{Henri Johnston$^2$}
\address{$^2$Department of Mathematics and Statistics, University of Exeter, Exeter EX4 4QF, United Kingdom}
\author{Hendrik W. Lenstra Jr.$^3$}
\address{$^3$Mathematisch Instituut, Universiteit Leiden, Postbus 9512, 2300 RA Leiden, The Netherlands}

\email{alex.bartel@glasgow.ac.uk, h.johnston@exeter.ac.uk}
\email{hwl@math.leidenuniv.nl}

\date{\today}

\begin{abstract}
The main purpose of the paper is to formulate a probabilistic model
for Arakelov class groups in families of number fields, offering a correction
to the Cohen--Lenstra--Martinet heuristic on ideal class groups.
To that end, we show that Chinburg’s $\Omega(3)$ conjecture implies
tight restrictions on the Galois module structure of oriented Arakelov class
groups. As a consequence, we construct a new infinite series of counterexamples
to the Cohen--Lenstra--Martinet heuristic, which have the
novel feature that their Galois groups are non-abelian.
\end{abstract}
\maketitle
\section{Introduction}\noindent

It has been an area of active research over the past few decades to understand
the distribution of class groups $\Cl_F$ of ``random'' algebraic number fields $F$.
Specifically, we let $K$ be a number field, and let $G$ be a finite group.
Let $\Lambda$ be the quotient of the group ring $\Z[\tfrac{1}{2\cdot \#G}][G]$
by the two-sided ideal generated by $\sum_{g\in G}g$. One studies the behaviour
of the $\Lambda$-module $\Lambda\otimes_{\Z[G]}\Cl_F$, as $F$ runs over a ``natural''
family $\cF$ of $G$-extensions of $K$.

An equally classical invariant that one attaches to $F$ is the
unit group $\cO_F^\times$ of the ring of integers $\cO_F$ of $F$,
viewed as a $\Z[G]$-module. Its isomorphism class, unlike that of
the class group, has only finitely many possibilities, as $F$ ranges
over $\cF$. The statistical properties of $\cO_F^\times$
have, however, been much less extensively studied.

In the present paper we make the case that, in this context, $\Cl_F$ and
$\cO_F^\times$ are most naturally studied in combination, since their distributions
need, by all appearances, not be independent.
Their dependence is best expressed by means of the \emph{Arakelov class group}.
It is a compact abelian group
attached to $F$, and we will recall its definition in Section \ref{sec:prelims}.
For number fields it plays the r\^ole that the Jacobian of a curve plays for
function fields over finite fields. It can be broken up into two pieces, one
being $\Cl_F$ and the other coming from $\cO_F^\times$, but in several ways
it is better than the sum of its parts. We find it convenient
to replace the Arakelov class group by its Pontryagin dual $\Ar_F$. This is a
finitely generated abelian group that fits into a short exact sequence
\[
0 \longrightarrow \Hom(\Cl_F,\Q/\Z) \longrightarrow \Ar_F \longrightarrow \Hom(\cO_F^\times,\Z) \longrightarrow 0.
\]
In other words, the torsion subgroup of $\Ar_F$ is the Pontryagin dual of the
class group, and its torsion-free part is the $\Z$-linear dual of
the unit group. This exact sequence, being canonically associated with $F$,
is an exact sequence of $\Z[G]$-modules.

Let $\G_0(\Lambda)$ denote the Grothendieck group of the category of 
finitely generated $\Lambda$-modules; see Section \ref{subsec:Groth-group-rings}
for the definition. 
Let $F/K$ be a Galois extension with Galois group $G$, let $S_{\infty}$ be the
$G$-set of Archimedean places of $F$, and let
$\Z^{S_\infty}$ be the corresponding permutation module over $\Z[G]$;
it is a property of $\cF$ that the isomorphism class of the $G$-set $S_{\infty}$
is independent of $F$ when $F \in \cF$.
The difference of the classes
$[\Lambda\otimes_{\Z[G]}\Ar_F]$ and $[\Lambda\otimes_{\Z[G]} \Z^{S_\infty}]$
in $\G_{0}(\Lambda)$ lies in the torsion subgroup $\G_{0}(\Lambda)_{\tors}$ of $\G_{0}(\Lambda)$, 
as can easily be deduced from Lemma \ref{lem:equiv-assertions-mc-torsion}.
This torsion subgroup is a finite abelian group, which can be thought of as a
``class group'' of $\Lambda$. 
The following result will be proven at the end of Section
\ref{sec:Chinburg} as a consequence of Proposition \ref{prop:omega-3-over-max-order-implies-conj:main}.

\begin{theorem}\label{thm:main}
  With the notation just introduced,
  suppose that Chinburg's $\Omega(3)$ conjecture, Conjecture \ref{conj:omega-3},
  holds for $F/K$. 
  Suppose, moreover, that for every prime number $p$ not dividing $2\cdot\#G$, each primitive $p$-th root of unity in $F$ is in $K$.
  Then the equality
  \[
    [\Lambda\otimes_{\Z[G]}\Ar_F] - [\Lambda\otimes_{\Z[G]} \Z^{S_\infty}] = 0
  \]
  holds in $\G_{0}(\Lambda)_{\tors}$.
\end{theorem}\noindent

In fact, we will prove the conclusion of Theorem \ref{thm:main} under a weaker hypothesis;
see Theorem \ref{thm:mainStronger}.

In Section \ref{sec:CL} we will define the families $\cF$ that we consider.
Together, Theorem \ref{thm:main} and Propositions \ref{prop:ZSinfty-constant} and \ref{prop:Ara}
imply that as $F$ ranges over $\cF$, the class
$[\Lambda\otimes_{\Z[G]}\cO_F^\times] - [\Lambda\otimes_{\Z[G]}\Cl_F]$
in $\G_0(\Lambda)$ is conjectured to be constant.
We have no reason to expect this to be true for either of the two terms individually.

There are cases in which the conclusion of Theorem \ref{thm:main} can be proven
unconditionally. This includes Galois extensions of $\Q$ of degree less than $112$
(see Proposition \ref{prop:112}), but also, more interestingly, a class of fields
that can be used to construct a new series of counterexamples to the
Cohen--Lenstra--Martinet heuristic \cite{MR756082,MR1037430} on class groups of
number fields. Informally, this heuristic reads as follows.

\begin{heuristic}[vague version]\label{he:introorig} Let $F$ vary in a natural
  family of Galois number fields. Then the Galois module $\Ar_F$, after inverting
  the ``bad'' prime numbers, behaves ``randomly'' with respect to a probability distribution
  that assigns to a suitable Galois module $M$ a probability weight that is inversely
  proportional to the ``size'' of the automorphism group of $M$.
\end{heuristic}\noindent
The precise version that we shall use will be formulated as Heuristic \ref{he:orig}
in Section~\ref{sec:CL}. The same section explains what we mean by the term ``natural
family''. It also allows the ring $\Lambda$ to be more general than the
ring considered above.

Heuristic \ref{he:orig} represents, in several respects, a corrected version of
the original Cohen--Lenstra--Martinet heuristic. Nevertheless, it is known to
be invalid, counterexamples for certain abelian Galois groups $G$ having been
provided in \cite[Theorem 1.1, Proposition 4.4]{MR4105790}. One of the main
results of the present paper is a new series of counterexamples, this time with
non-abelian~$G$. The other main achievement is a proposed correction to the heuristic.

The new counterexamples make use of groups $G$ of order $2^p\cdot p$, where $p$
is an odd prime number. Their abelianisations $G/G'$, which are cyclic of order $2p$,
coincide with the groups on which the abelian counterexamples in \cite{MR4105790}
depend. Our groups have centres $Z$ of order $2$, and writing $Z = \langle \gamma\rangle$,
we shall make use of the ring $\Lambda = \Z[\tfrac{1}{2p}][G]/(1+\gamma)$. In Section \ref{sec:disproof}
we prove the following theorem.

\begin{theorem}\label{thm:introdisproof}
For infinitely many odd prime numbers $p$ there is a group $G$ with the
properties just listed such that the following is true. With $\Lambda$ as just
defined, the group $\G_0(\Lambda)_{\tors}$ is non-trivial, whereas there does exist a natural
family of $G$-extensions of $\Q$ such that for all members $F$ of the family the
class of $\Lambda\otimes_{\Z[G]} \Cl_F$ in $\G_0(\Lambda)_{\tors}$ is trivial.
\end{theorem}\noindent
The families in Theorem \ref{thm:introdisproof} necessarily violate
Heuristic \ref{he:orig}, since the latter
would imply equidistribution of $\Lambda\otimes_{\Z[G]} \Cl_F$ in $\G_0(\Lambda)_{\tors}$
as $F$ ranges over the family. This is proven in Section \ref{sec:disproof},
to which we also refer for more information on the groups and families appearing
in Theorem \ref{thm:introdisproof}.

The probability weight that is
inversely proportional to the ``size'' of the automorphism group, as referred to
in Heuristic \ref{he:introorig}, reflects an
attractive feature of the Arakelov class groups. The general principle behind many
heuristics is that algebraic objects in natural families tend to be ``as random as they can be'',
with respect to a probability distribution that assigns to an algebraic object $X$
a probability weight that is proportional to $1/\# \Aut X$. The original
Cohen--Lenstra--Martinet heuristic did, initially, look like an exception to the
rule just mentioned, but this changed when it was reformulated in terms of $\Ar_F$.
This is discussed in more detail in \cite{MR4105790}; see also Section \ref{sec:CL} below.

Theorem \ref{thm:main} restricts how random Arakelov class groups can be. We propose
a correction to Heuristic \ref{he:introorig} that takes this restriction into account.

\begin{heuristic}[vague version]\label{he:intronew}
  Let $F$ vary in a natural family of Galois number fields. Then the Galois module
  $\Ar_F$, after inverting the ``bad'' prime numbers,
  behaves ``randomly'' with respect to a probability distribution that assigns
  to a suitable Galois module $M$ a probability weight that is inversely proportional
  to the ``size'' of the automorphism group of $M$, restricted to those modules
  that satisfy equation \eqref{eq:ArAndArtilde} with
  $M$ in place of $\Ar_F$.
\end{heuristic}\noindent
A precise version of Heuristic \ref{he:intronew} will be formulated as Heuristic \ref{he:main} in Section~\ref{sec:CL}.

Note that the only difference between Heuristics \ref{he:introorig} and \ref{he:intronew}
is the reference to \eqref{eq:ArAndArtilde}. It expresses
the restriction that Chinburg's $\Omega(3)$ conjecture imposes on the class of $\Ar_F$
in the class group of $\Lambda$ as a consequence of Theorem \ref{thm:main}. 
It is also important to point out in which way Heuristic \ref{he:intronew} differs from Conjecture 1.5 formulated in \cite{MR4105790}.
The latter conjecture is only concerned with the local structure of $\Ar_F$ at a \emph{finite} set of prime numbers;
in that case, the class group of $\Lambda$ is trivial, so that Chinburg's $\Omega(3)$ conjecture
imposes no restriction. On the other hand, Heuristic \ref{he:intronew} considers almost
\emph{all} prime numbers, and 
in this global situation $\Lambda$ may have a non-trivial class group.
In particular, one should be able to extract explicit information on the distribution
of $\cO_F^\times$ as a Galois module from Heuristic \ref{he:intronew}, which is not possible
with Conjecture 1.5 of \cite{MR4105790}.

The structure of the paper is as follows. After some preliminaries in Section \ref{sec:prelims},
we formulate in Section \ref{sec:CL} the old and the new heuristics. Some basic material
on Grothendieck groups of orders is treated in Section~\ref{sec:duals}. Section~\ref{sec:Arakelov}
is devoted to Arakelov class groups as Galois modules, and Section~\ref{sec:Chinburg} to the
implications of Chinburg's $\Omega(3)$ conjecture for these Galois modules. 
In Section~\ref{sec:known}
we collect some cases in which the conclusion of Theorem \ref{thm:main} is known,
and use these in Section~\ref{sec:disproof}
to construct a new series of counterexamples to the Cohen--Lenstra--Martinet heuristic.
\begin{acknowledgements}
We would like to thank Ted Chinburg, Aurel Page, and Peter Stevenhagen for useful
conversations, and Andreas Nickel and an anonymous referee for helpful comments
on earlier drafts of the manuscript. The first named author gratefully acknowledges
financial support through EPSRC Fellowship EP/P019188/1, `Cohen--Lenstra heuristics,
Brauer relations, and low-dimensional manifolds'.
The second named author gratefully acknowledges financial support through EPSRC First
Grant EP/N005716/1 `Equivariant Conjectures in Arithmetic'.
\end{acknowledgements}

\section{Preliminaries}\label{sec:prelims}\noindent
In this section we recall material that we will use for the formulation of Heuristic \ref{he:main}. 
In particular, we recall from \cite{Schoof} the definition of the Arakelov
class group and of the oriented Arakelov class group of a number field.

\subsection{Pontryagin duality}\label{subsec:Pontryagin}
We briefly recall some facts on Pontryagin duality and refer the reader to
\cite[Chapter 1, \S 1]{MR2392026} for a more detailed overview. If $A$ and $B$
are abelian topological groups, then $\Hom_{\cts}(A,B)$ denotes the group of continuous
group homomorphisms from $A$ to $B$. Let $\cC$
be the category of Hausdorff locally compact abelian topological groups.
If $A$ is an object of $\cC$, then its Pontryagin dual is defined to 
be $\Hom_{\cts}(A,\R/\Z)$. This defines an anti-equivalence
of $\cC$ with itself, of which the square is isomorphic to the identity
functor. It induces an anti-equivalence between the full subcategories of
compact abelian groups and of discrete abelian groups. If $M$ is a finitely
generated abelian group, then $\Hom_{\cts}(M \otimes_{\Z} (\R/\Z),\R/\Z)$
is canonically isomorphic to $\Hom(M,\Z)$. 
In particular, we have $\Hom_{\cts}(\R/\Z,\R/\Z) \cong \Z$.

\subsection{The (oriented) Arakelov class group}\label{sec:Arak}
Let $F$ be a number field.
Let $\Id_F$ be the group of fractional ideals of the ring of integers $\cO_F$ of $F$,
let $S_{\infty}$ denote the set of Archimedean places of $F$, and let $F_\R$
denote the \'etale $\R$-algebra $F\otimes_{\Q}\R=\prod_{w\in S_{\infty}}F_w$,
where $F_w$ denotes the completion of $F$ at $w$. We have canonical maps
$\fN\colon \Id_F\to \R_{>0}$ and $\AbsNm\colon F_{\R}^\times\to \R_{>0}$, the first given
by the ideal norm, and the second given by the absolute value of the $\R$-algebra
norm. Let $\Id_F\times_{\R_{>0}} F_{\R}^\times$ denote the fibre product with
respect to these maps. The \emph{oriented Arakelov class group} $\PicTilde{F}$
of $F$ is defined as the cokernel of the map $F^\times\to \Id_F\times_{\R_{>0}} F_{\R}^\times$
that sends $\alpha\in F^\times$ to $(\alpha\cO_F,\alpha)$.
It follows from Dirichlet's unit theorem and the finiteness of the
class group of $\cO_F$, that this is a compact abelian group.

For every $w\in S_{\infty}$ we have a direct product decomposition
$F_{w}^\times\cong \R_{>0} \times \comp(F_w^\times)$, where $\comp(F_w^\times)$
is the maximal compact subgroup of $F_w^\times$, which is equal to $\{\pm 1\}$
if $w$ is real, and to the circle group in $F_w$ if $w$ is complex. The maximal
compact subgroup $\comp(F_{\R}^\times)=\prod_{w\in S_{\infty}}\comp(F_w^\times)$
of $F_{\R}^\times$ is contained in the kernel of the map $\AbsNm$. Define the
\emph{Arakelov class group} $\Pic^0_F$ of $F$ to be the quotient of
$\PicTilde{F}$ by the image of
$\{1\}\times \comp(F_{\R}^\times)\subset \Id_F\times_{\R_{>0}}F_{\R}^\times$ in
$\PicTilde{F}$.

As in the introduction, we write $\Ar_F$ for the Pontryagin dual of
$\Pic_F^0$, and we write $\Ara{F}$ for the Pontryagin dual of $\PicTilde{F}$.
Let $\mu_{F}$ denote the group of roots of unity in $F$. 
Note  that $\Ar_F$, $\Ara{F}$, and $\mu_{F}$ are all
finitely generated left $\Z[\Aut F]$-modules; for $\Ar_F$ this will follow
from the next result, and for $\Ara{F}$ this follows from \cite[Proposition 5.3]{Schoof}.

\begin{proposition}\label{prop:SES-for-AF}
There is a short exact sequence  
\[
0 \longrightarrow \Hom(\Cl_F,\Q/\Z) \longrightarrow \Ar_F \longrightarrow \Hom(\cO_F^\times / \mu_{F},\Z) \longrightarrow 0.
\]
\end{proposition}

\begin{proof}
By \cite[Proposition 2.2]{Schoof} there is an exact sequence
\[
  0 \longrightarrow (\cO_{F}^\times/\mu_{F})\otimes_{\Z}(\R/\Z) \longrightarrow \Pic_{F}^{0} \longrightarrow \Cl_{F} \longrightarrow 0. 
\] 
The desired result follows by taking the Pontryagin dual of this sequence
and recalling from Section \ref{subsec:Pontryagin} 
that we have an isomorphism $\Hom_{\cts}((\cO_{F}^\times/\mu_{F})\otimes_{\Z}(\R/\Z),\R/\Z) \cong \Hom(\cO_F^\times/\mu_F,\Z)$.
\end{proof}

\begin{remark}\label{rmk:similar-SES}
Note that the canonical map 
$\Hom(\cO_F^\times / \mu_{F},\Z)\to \Hom(\cO_F^\times,\Z)$
is an isomorphism.
\end{remark}

\subsection{Modules and Grothendieck groups}\label{subsec:Groth-group-rings}
Henceforth all modules will be assumed to be left modules unless stated otherwise. 
If $G$ is a group, $S$ is a finite $G$-set, and $R$ is a ring, in the remainder
of the paper $R^S$ will denote the free $R$-module on the set $S$ with the
induced $R$-linear $R[G]$-action. We will refer to such $R[G]$-modules as
\emph{permutation modules}.

Recall that for a ring $T$, the Grothendieck group $\G_{0}(T)$ of the category of 
finitely generated $T$-modules is the additive group generated 
by expressions $[M]$, one for each isomorphism class of finitely generated 
$T$-modules $M$, with a relation $[L]+[N]=[M]$ whenever there exists a short 
exact sequence 
\[
0 \longrightarrow L \longrightarrow M \longrightarrow N \longrightarrow 0
\]
of finitely generated $T$-modules.

If $P$ is a set of prime numbers, then we define
\[
\textstyle{\Z_{(P)}=\{ a/b : a, b\in \Z, b \not \in \bigcup_{p\in P\cup\{0\}}p\Z \}.}
\]

If $T\to T'$ is a ring homomorphism such that $T'$ is a flat right $T$-module,
then the functor $T' \otimes_T \bullet$ from the category of finitely generated $T$-modules
to that of finitely generated $T'$-modules
induces a group homomorphism $\G_0(T)\to \G_0(T')$. The following two examples
of this construction will be relevant to us. If $G$ is a finite group and $P$
is a set of prime numbers, then the flat localisation map $\Z\to \Z_{(P)}$
induces a group homomorphism $\G_0(\Z[G])\to \G_0(\Z_{(P)}[G])$.
Moreover, if we have a direct product decomposition $T\cong U\times W$ of rings, then
the right $T$-module $W$ is projective, and in particular flat, so the quotient
map $T\to W$ induces a group homomorphism $\G_0(T)\to \G_0(W)$.

\section{Cohen--Lenstra--Martinet heuristic}\label{sec:CL}\noindent
In this section we propose a correction of the Cohen--Lenstra--Martinet
heuristic \cite{MR756082,MR1037430}.
The notation and assumptions introduced in the next three paragraphs
will remain in force throughout this section.

Let $G$ be a finite group, let $P$ be a set of prime numbers not dividing
$2\cdot \#G$, let the $\Q$-algebra $A$ be a quotient of $\Q[G]$ by a two-sided ideal containing
$\sum_{g\in G}g$, and let $\Lambda$ be the image of $\Z_{(P)}[G]$ in $A$;
note that the ring $\Lambda$ in the introduction is a special case of this. 
Next, let $V$ be a finitely generated $\Q[G]$-module, and, for brevity, set
$V_A=A\otimes_{\Q[G]}V$. Let $\cM_V$ be a set of finitely generated
$\Lambda$-modules $M$ that satisfy $A\otimes_{\Lambda} M \cong_A V_A$, and
with the property that for every finitely generated $\Lambda$-module $M'$
satisfying $A\otimes_{\Lambda} M'\cong_A V_A$ there exists a unique $M\in \cM_V$
such that $M'\cong M$. Note that the set $\cM_V$ is countable.
If $M$ is a finitely generated $\Lambda$-module satisfying
$A\otimes_\Lambda M\cong_A V_A$, and $f$ is a function defined on $\cM_V$,
then we write $f(M)$ for the value of $f$ on the unique element of $\cM_V$
that is isomorphic to $M$. 
In \cite{MR3705226} it was shown that there is a unique ``automorphism index''
function $\ia\colon \cM_V\times\cM_V\to \Q_{>0}$ that behaves, in a precise
sense explained in \cite[Theorem 1.1]{MR3705226}, like
$(L,M)\mapsto \frac{\#\Aut M}{\#\Aut L}$, even when the automorphism groups
of $M$ and of $L$ are infinite. Fix $M\in \cM_V$. If $\cN$ is a subset of $\cM_V$
and $X$ is a positive real number, let $\cN_X$ be the finite set of all $L\in \cN$
whose torsion subgroup $L_{\tors}$ has order less than $X$. For $\cN\subset \cM_V$ and for a
function $f\colon \cN\to \C$, define the expected value of $f$ on $\cN$ by
\[ 
  \bE_{\cN}(f)=\lim_{X\to \infty}\left(\sum_{L\in \cN_X}\ia(L,M)f(L)\Big/\sum_{L\in\cN_X} \ia(L,M)\right)
\]
when the limit exists. One of the
defining properties of the function $\ia$ is that for all $L$, $M$, and $N\in \cM_V$
we have $\ia(L,M)\ia(M,N)=\ia(L,N)$, whence it follows that whether or not
$\bE_{\cN}(f)$ is defined is independent of the choice of $M$, and so is
its value when it is defined.
\begin{remark}\label{rmrk:expvaluequo}
  Expected values behave well under passing to quotients in the following sense.
  Let $G_1\to G_2$ be a surjective group homomorphism, let $P$ be a set
  of prime numbers not dividing $2\cdot \#G_1$, let $A_2$ be a quotient of
  $\Q[G_2]$ as above, and let $V_2$ be a finitely generated $\Q[G_2]$-module;
  let $A_1$ be the same as $A_2$ but viewed as a quotient of $\Q[G_1]$,
  and let $V_1$ be the same as $V_2$ but viewed as a $\Q[G_1]$-module. 
  For $i\in \{1,2\}$, let $\Lambda_i$ be the image of $\Z_{(P)}[G_i]$ in $A_i$.
  Note that in particular 
  the map $G_1\to G_2$ induces a ring isomorphism $\Lambda_1\to \Lambda_2$.
  For $i\in \{1,2\}$, define sets $\cM_{V_i}$ of
  $\Lambda_i$-modules as above. For brevity, write $\cM_i=\cM_{V_i}$ for $i\in \{1,2\}$.
  Let $f_2\colon \cM_{2}\to \C$ be a function,
  and let $f_1\colon \cM_{1}\to \C$ be given by $M\mapsto f_2(\Lambda_2\otimes_{\Lambda_1}M)$.
  Then one has
  $
    \bE_{\cM_1}(f_1) = \bE_{\cM_2}(f_2).
  $
\end{remark}\noindent
Let $K$ be a number field, and let $\Kbar$ be an algebraic closure
of $K$. Given a pair $(F,i)$, where $F\subset \Kbar$ is a Galois extension of $K$
and $i$ is an isomorphism between the Galois group of $F/K$
and $G$, we view $\Gal(F/K)$-modules as $G$-modules via $i$.
Let $\cF$ be the set of all such pairs
$(F,i)$
for which $F$ contains no primitive $p$-th root of unity for any prime number $p\in P$,
and for which there is an isomorphism $\Q\otimes_{\Z} \cO_F^\times\cong V$
of $\Q[G]$-modules. Assume that $\cF$ is infinite. Such an $\cF$ is what we
called a ``natural family'' in the introduction. 
Note that this family is a special case of the families considered in \cite[\S 2]{MR4105790}. 

For $(F,i)\in \cF$, let $c_{F/K}$ be the ideal norm of the product
of the prime ideals of $\cO_K$ that ramify in $F/K$. For a positive real number
$B$, let $\cF_{c\leq B}=\{(F,i)\in \cF: c_{F/K}\leq B\}$. 
The following version of the Cohen--Lenstra--Martinet heuristic is a variant
of \cite[Heuristic 2.1]{MR4105790} phrased in terms of Arakelov class groups.
It differs in several ways from
the heuristic formulated in \cite{MR756082,MR1037430}, but none
of those differences shall concern us in the present paper.
\begin{heuristic}\label{he:orig}
  Let $f$ be a ``reasonable''
  $\C$-valued function on $\cM_V$. Then the limit
  $$
  \lim_{B\to\infty} \frac{\sum_{(F,i)\in \cF_{c\leq B}}f(\Lambda\otimes_{\Z[G]}\Ar_F)}{\#\cF_{c\leq B}}
  $$
  exists, and is equal to $\bE_{\cM_V}(f)$.
\end{heuristic}\noindent
The notion of a reasonable function is left intentionally vague. The functions
considered in \cite{CMII} give rise to many examples of presumably reasonable
functions on $\cM_V$ that factor through $M\mapsto M_{\tors}$. An example of a
function not of that form that we would consider reasonable, and which depends
on the Galois module structure of both the class group and the unit group of
the ring of integers, is $M\mapsto \#\Hom_{\Lambda}(M/M_{\tors},M_{\tors})$.

If the set $P$ is infinite, then Conjecture \ref{conj:main-for-CL}
can be an obstruction to the conclusions of Heuristic \ref{he:orig}.
For example it was shown in \cite[\S 4]{MR4105790}, as a consequence of a
proven special case of Conjecture \ref{conj:main-for-CL}, that the conclusion
of Heuristic \ref{he:orig} does not, in general, hold for functions of the
form $M\mapsto \chi([M])$, where $\chi\colon \G_{0}(\Lambda)\to \C^\times$
is a homomorphism of finite order. In \cite{MR4105790} a corrected heuristic
was proposed in which $P$ was assumed to be finite.
In the remainder of the section, we formulate a
Cohen--Lenstra--Martinet heuristic without the hypothesis that $P$ be finite.

\begin{proposition}\label{prop:ZSinfty-constant}
  Let $(F,i)$ and $(F',i')\in \cF$, and let $S_{\infty}$ and $S_{\infty}'$
  be the sets of Archimedean places of $F$ and $F'$, respectively. 
  Then the equality 
  \[
    [\Lambda\otimes_{\Z[G]} \Z^{S_{\infty}}] = [\Lambda\otimes_{\Z[G]} \Z^{S_{\infty}'}]
  \]
  holds in $\G_{0}(\Lambda)$.
\end{proposition}

\begin{proof}
  By definition of the family $\cF$, we have an isomorphism
  $\Q\otimes_{\Z}\cO_F^\times\cong
  \Q\otimes_{\Z}\cO_{F'}^\times$ of $\Q[G]$-modules.
  By a theorem of Herbrand there is an isomorphism
  $(\Q\otimes_{\Z}\cO_F^\times) \oplus \Q \cong\Q^{S_{\infty}}$ of $\Q[G]$-modules,
  see for example \cite[Chapter I, 4.3]{MR782485}, and similarly for $\cO_{F'}^\times$.
  Thus, we have an isomorphism $\Q^{S_{\infty}}\cong \Q^{S_{\infty}'}$.
  
  Since all point stabilisers for $S_{\infty}$ and $S_{\infty}'$ are
  inertia groups at Archimedean places, they are all cyclic.
  It follows from Artin's induction theorem (e.g. by combining
  \cite[\S 13.1, Corollary 1 and Theorem 30]{Serre}
  and comparing dimensions) that if $S$ and $S'$ are finite $G$-sets with
  cyclic point stabilisers such that there is an isomorphism
  $\Q^S\cong \Q^{S'}$ of $\Q[G]$-modules,
  then the $G$-sets $S$ and $S'$ are isomorphic.
  In particular, there is then an isomorphism
  $\Z^S\cong \Z^{S'}$
  of $\Z[G]$-modules. The result follows by applying this observation
  to the $G$-sets $S_{\infty}$ and $S_{\infty}'$.
\end{proof}\noindent
We define $C(\cF)$ to be the common class of $\Lambda\otimes_{\Z[G]} \Z^{S_{\infty}}$ in $\G_{0}(\Lambda)$
for all $(F,i)\in \cF$, where $S_{\infty}$ is the $G$-set of Archimedean places of $F$.

\begin{heuristic}\label{he:main}
  Let $\cN=\{M\in \cM_V: [M]=C(\cF)\text{ in }\G_{0}(\Lambda)\}$, and let $f$ be a ``reasonable''
  $\C$-valued function on $\cM_V$.
  Then the limit
  \[
  \lim_{B\to\infty} \frac{\sum_{(F,i)\in \cF_{c\leq B}}f(\Lambda\otimes_{\Z[G]}\Ar_F)}{\#\cF_{c\leq B}}
  \]
  exists, and is equal to $\bE_{\cN}(f)$.
\end{heuristic}

\section{Grothendieck groups of orders}\label{sec:duals}\noindent
In this section we review some standard facts about
Grothendieck groups of orders,
and examine the effect
of some duality operations upon these Grothendieck groups.

Let $R$ be a Dedekind domain and let $k$ be the field of
fractions of $R$.
An \emph{$R$-order} is an $R$-algebra that is finitely generated and projective as
an $R$-module.
For example if $G$ is a finite group, then the group
ring $\Lambda=R[G]$ is an $R$-order.

Let $\Lambda$ be an $R$-order.
A finitely generated $\Lambda$-module that is projective over $R$ will be
referred to as a \emph{$\Lambda$-lattice}.
Let $\G_{0}^{R}(\Lambda)$ denote the Grothendieck group of the category of
$\Lambda$-lattices. By definition, $\G_{0}^{R}(\Lambda)$ is the additive group generated 
by expressions $[M]$, one for each isomorphism class of $\Lambda$-lattices 
$M$, with a relation $[L]+[N]=[M]$ whenever there exists a short exact sequence 
\[
0 \longrightarrow L \longrightarrow M \longrightarrow N \longrightarrow 0
\] of $\Lambda$-lattices.
Recall from Section \ref{subsec:Groth-group-rings} that 
if $T$ is a ring, then we similarly define the Grothendieck group $\G_0(T)$
of the category of finitely generated $T$-modules by replacing, in the above
definition, ``$\Lambda$-lattice'' by ``finitely generated $T$-module''.
By \cite[Theorem (38.42)]{MR892316}, the inclusion of the category
of $\Lambda$-lattices into the category of all finitely generated
$\Lambda$-modules induces a canonical isomorphism $\G_{0}^{R}(\Lambda)\cong \G_{0}(\Lambda)$.

Let $\Lambda^{\op}$ denote the opposite ring of $\Lambda$.
If $M$ is a $\Lambda$-lattice, then $M^*=\Hom_R(M,R)$ is a $\Lambda^{\op}$-lattice.
This defines a contravariant functor from the category of $\Lambda$-lattices
to the category of $\Lambda^{\op}$-lattices, given on objects by $M\mapsto M^*$
for every $\Lambda$-lattice $M$, and on morphisms by $f\mapsto (\nu\mapsto \nu\circ f)\in M^*$
for every morphism $f\colon M\to N$ of $\Lambda$-lattices and every $\nu\in N^*$.
This functor is easily seen to be exact, and to induce a group isomorphism
$\G_0^R(\Lambda)\to\G_0^R(\Lambda^{\op})$, and hence an isomorphism
$$
  \involGrothen\colon \G_0(\Lambda)\to \G_0(\Lambda^{\op}).
$$

For a $\Lambda$-module $N$, we define $N^{\lor}=\Hom_R(N,k/R)$, which is also a
$\Lambda^{\op}$-module. If $N$ is a finitely generated $\Lambda$-module that
is $R$-torsion, then $N^{\lor}$ is finitely generated over $\Lambda^{\op}$
and $R$-torsion.

In the special case that $\Lambda=R[G]$, where $G$ is a finite group,
the ring $\Lambda$ is equipped with an involution $\iota$
induced by $g\mapsto g^{-1}$ for all $g\in G$. If $M$ is a finitely generated
$R[G]$-module, then we view the $R[G]^{\op}$-module $M^{*}$
as a finitely generated $R[G]$-module via $\iota$, and we view the map $\involGrothen$
as an automorphism of $\G_0(R[G])$. Specialising further to $R=\Z$, if $A$ is a
Hausdorff locally compact abelian topological group
on which $G$ acts by continuous automorphisms, then we view its Pontryagin dual
$\Hom_{\cts}(A,\R/\Z)$ as a $\Z[G]$-module via $\iota$. If $N$ is a
$\Z[G]$-module of finite cardinality,
then so is $\Hom_{\cts}(N,\R/\Z)=\Hom_{\Z}(N,\Q/\Z)=N^{\lor}$.
\begin{proposition}\label{prop:dual}
  Let $N$ be a finitely generated $\Lambda$-module that is $R$-torsion. Then the equality
  \[
  [N^{\lor}]=-\involGrothen[N]
  \] 
  holds in $\G_{0}(\Lambda^{\op})$.
\end{proposition}

\begin{proof}
  Since $R$ is a Dedekind domain, every $\Lambda$-submodule of a $\Lambda$-lattice
  is itself a $\Lambda$-lattice. Hence there exists a presentation
  $0\to M_1\to M_2\to N\to 0$ of $N$ by $\Lambda$-lattices, so that
  $[N] = [M_2]-[M_1]$ in $\G_{0}(\Lambda)$. We claim that $N^\lor$ is canonically
  isomorphic as a $\Lambda^{\op}$-module to $M_1^*/M_2^*$.

  Since $M_1$, $M_2$ are projective over $R$, applying the functors $\Hom_R(M_i,\bullet)$
  for $i=1$, $2$ to the short exact sequence
  \[
    0 \longrightarrow R \longrightarrow k \longrightarrow k/R \longrightarrow 0
  \]
  yields the commutative diagram with exact rows
  \[
    \xymatrix{
      0 \ar[r] & \Hom_R(M_2,R) \ar[r]\ar[d] & \Hom_R(M_2,k) \ar[r]\ar[d] & \Hom_R(M_2,k/R) \ar[r]\ar[d] & 0\\
      0 \ar[r] & \Hom_R(M_1,R) \ar[r] & \Hom_R(M_1,k) \ar[r] & \Hom_R(M_1,k/R) \ar[r] & 0,\\
    }
  \]
  where the vertical maps are induced by the injection $M_1\to M_2$.
  Of these, the middle map $\Hom_R(M_2,k)\to \Hom_R(M_1,k)$ is an isomorphism.
  Indeed, it is the $k$-linear dual of the map $k\otimes_R M_1\to k\otimes_R M_2$,
  which is clearly an isomorphism, since the cokernel $N$ of $M_1\to M_2$ is $R$-torsion.
  The snake lemma therefore gives an isomorphism of right $\Lambda$-modules
  from the kernel of $\Hom_R(M_2,k/R)\to \Hom_R(M_1,k/R)$ to the cokernel of
  $\Hom_R(M_2,R)\to \Hom_R(M_1,R)$. Since $\Hom_R(\bullet,k/R)$
  is left exact, that kernel is exactly $N^\lor$, while the cokernel is precisely $M_1^*/M_2^*$, as claimed.
  The proposition immediately follows.
\end{proof}

\section{Oriented Arakelov class groups as Galois modules}\label{sec:Arakelov}\noindent
In this section we prove some properties of oriented Arakelov class
groups as Galois modules, and formulate the main working hypothesis that motivates the
statistical heuristic in Section \ref{sec:CL}.

\begin{lemma}\label{lem:tauv}
Let $F/K$ be a finite Galois extension of number fields, let $G$ be the Galois
group, let $S_{\infty}$ be the set of Archimedean places of $F$,
and let $d$ be the degree of $K$ over $\Q$. Then the equality
\[
  [\Hom_{\cts}(\comp(F_{\R}^\times),\R/\Z)] = d\cdot [\Z[G]] - [\Z^{S_{\infty}}]
\]
holds in $\G_0(\Z[G])$, where $\comp(F_{\R}^\times)$ is as defined in Section \ref{sec:Arak}.
\end{lemma}
\begin{proof}
If $v$ is an Archimedean place of $K$, let $I_v\subset G$ denote an inertia
subgroup at $v$, and let $\tau_v$ be a $\Z[I_v]$-module defined as follows:
if $v$ is real and $I_v$ is the trivial group, then $\tau_v=\F_2$;
if $v$ is real and $I_v$ has order $2$, then $\tau_v$ is free over $\Z$
of rank $1$, and with the generator of $I_v$ acting by $-1$; and if $v$
is complex, so that $I_v$ is necessarily trivial, then $\tau_v=\Z$.
Then it is easy to see that we have an isomorphism
\begin{eqnarray*}
  \Hom_{\cts}(\comp(F_{\R}^\times),\R/\Z)\cong \bigoplus_v\Ind^G_{I_v}\tau_v
\end{eqnarray*}  
of $\Z[G]$-modules, where the direct sum runs over the Archimedean places of $K$,
and $\Ind^G_{I_v}$ denotes induction from $I_v$ to $G$. 

If $v$ is a real place of $K$ such that $I_v$ is trivial,
then the exact sequence
\[
0 \longrightarrow \Ind^G_{I_v}\Z\stackrel{\times 2}{\longrightarrow}\Ind^G_{I_v}\Z \longrightarrow \Ind^G_{I_v}\F_2 \longrightarrow 0
\]
shows that $[\Ind^G_{I_v}\tau_v] =0$ in $\G_{0}(\Z[G])$. We deduce
that for all Archimedean places $v$ of $K$ we have
\begin{eqnarray}\label{eq:tauv}
[\Ind^G_{I_v}\tau_v] = \delta_v[\Z[G]] - [\Z^{G/I_v}],
\end{eqnarray}
where $\delta_v=1$ if $v$ is real, and $\delta_v=2$ is $v$ is complex.
The result follows by summing \eqref{eq:tauv} over all Archimedean
places $v$ of $K$.
\end{proof}\noindent
For every permutation $\Z[G]$-module $M$, we have
$M\cong M^{*}$, where $M^*$ is as defined in Section \ref{sec:duals}
with $R=\Z$, and therefore $\involGrothen[M] = [M]$ in $\G_{0}(\Z[G])$.
We will repeatedly use this observation.
If $S$ is a set of places of $F$ containing all Archimedean places,
let $\cO_{F,S}$ denote the ring of $S$-integers in $F$, 
let $\cO_{F,S}^{\times}$ denote its unit group, and let $\Cl_{F,S}$
denote its class group. If $S$ is $G$-stable, then $\cO_{F,S}^\times$ and
$\Cl_{F,S}$ are $\Z[G]$-modules.
Recall that $\Ar_F$ and $\Ara{F}$ denote the Pontryagin duals of
$\Pic_F^0$ and $\PicTilde{F}$, respectively, and $\mu_F$ denotes the group
of roots of unity in $F$.

\begin{proposition}\label{prop:Ara}
  Let $F/K$ be a finite Galois extension of number fields, let $G$ be the Galois group,
  let $d$ be the degree of $K$ over $\Q$, and let $S$ be a finite $G$-stable
  set of places of $F$ containing all Archimedean places. Then $\Ar_F$ and $\Ara{F}$
  are finitely generated $\Z[G]$-modules. Moreover the equalities
  \begin{align*}
    [\Ar_F] & = \involGrothen[\cO_F^\times/\mu_F] - \involGrothen[\Cl_F],\\
    [\Ara{F}] & = d\cdot [\Z[G]] - [\Z^S]+ \involGrothen[\cO^\times_{F,S}]-\involGrothen[\Cl_{F,S}]
  \end{align*}
  hold in $\G_{0}(\Z[G])$.
\end{proposition}

\begin{proof}
It was already noted in Section \ref{sec:Arak} that $\Ar_F$ and $\Ara{F}$ are finitely generated 
$\Z[\Aut F]$-modules, and this immediately implies the first assertion. 
The expression for $[\Ar_F]$ follows from Propositions \ref{prop:SES-for-AF} and \ref{prop:dual}. 

The rest of the proof is devoted to the derivation of the expression for $[\Ara{F}]$.
Let $S_{\f}=S\setminus S_{\infty}$ denote the set of non-Archimedean places in $S$.
  The subgroup of $\Id_F$ generated by the prime ideals corresponding
  to the places in $S_{\f}$ is free abelian
  on the set $S_{\f}$. Below, when we write $\Z^{S_{\f}}$,
  we will mean that subgroup.
  Let $\Id_{F,S}$ be the quotient of
  $\Id_F$ by the subgroup $\Z^{S_{\f}}$. It is naturally isomorphic to the group
  of fractional ideals of $\cO_{F,S}$.
  The preimage of
  $$
  \Z^{S_{\f}}\!\times_{\R_{>0}}F_{\R}^\times\subset\Id_F\times_{\R_{>0}}F_{\R}^\times
  $$
  under the inclusion map $F^\times\to \Id_F\times_{\R_{>0}}F_{\R}^\times$
  is $\cO_{F,S}^{\times}$.
  There is thus a commutative diagram of $\Z[G]$-modules with exact rows and columns
  $$
  \xymatrix{
    & 0\ar[d] & 0\ar[d] & 0\ar[d] & \\
    0\ar[r] & \cO_{F,S}^\times \ar[d]\ar[r] &
    \Z^{S_{\f}}\!\times_{\R_{>0}}F_{\R}^\times \ar[r]\ar[d] & T_{F,S} \ar[d]\ar[r] & 0\\
    0 \ar[r] & F^\times \ar[r]\ar[d] & \Id_F\times_{\R_{>0}}F_{\R}^\times\ar[r]\ar[d] & \widetilde{\Pic}_{F}^0 \ar[r]\ar[d] & 0\\
    0 \ar[r] & F^\times/\cO_{F,S}^\times \ar[r]\ar[d] & \Id_{F,S} \ar[r]\ar[d] & \Cl_{F,S} \ar[r]\ar[d] & 0\\
               & 0 & 0 & 0, & \\
  }
  $$
  where $T_{F,S}$ is defined by the exactness of the last column, and the exactness
  of the first row follows from the snake lemma.

The group $\Z^{S_{\f}}\!\times_{\R_{>0}}F_{\R}^\times$ can be explicitly described as follows: we have
\begin{eqnarray*}
    \Z^{S_{\f}}\!\times_{\R_{>0}}F_{\R}^\times =
    \Big\{\big((a_{\fp})_{\fp},b\big)\in \Z^{S_{\f}}\!\times F_{\R}:
    \prod_{\fp\in S_{\f}}\!\left(\fN\fp\right)^{a_{\fp}} = \AbsNm(b)\Big\},
\end{eqnarray*}
where $\fN$ and $\AbsNm$ are as defined in Section \ref{sec:Arak}.
That group naturally embeds into
$\R^{S_{\f}}\!\times_{\R_{>0}}F_{\R}^\times$, where the fibre product is taken
with respect to the map
\begin{eqnarray*}
  \R^{S_{\f}} & \longrightarrow & \R_{>0},\\
  (a_{\fp})_{\fp\in S_{\f}} & \mapsto & \prod_{\fp\in S_{\f}}\!\left(\fN\fp\right)^{a_{\fp}},
\end{eqnarray*}
and to the same map $\AbsNm\colon F_{\R}^\times\to \R_{>0}$
as before, so that we have an exact sequence
\begin{eqnarray}\label{eq:S-torus}
  0 \longrightarrow T_{F,S} \longrightarrow \frac{\R^{S_{\f}}\!\times_{\R_{>0}}F_{\R}^\times}{\cO_{F,S}^\times}
  \longrightarrow (\R/\Z)^{S_{\f}} \longrightarrow 0.
\end{eqnarray}

The preimage of $\{0\}\times \comp(F_{\R}^\times)\subset
\Z^{S_{\f}}\!\times_{\R_{>0}}F_{\R}^\times$ in $\cO_{F,S}^\times$ is
$\mu_F$. We therefore deduce from 
Dirichlet's ($S$-)unit theorem \cite[Chapter V, \S 1]{MR1282723}, that the
middle term of \eqref{eq:S-torus} is an extension of the form
\begin{eqnarray}\label{eq:Dirichlet}
  0 \longrightarrow \comp(F_{\R}^\times)/\mu_F \longrightarrow \frac{\R^{S_{\f}}\!\times_{\R_{>0}}F_{\R}^\times}{\cO_{F,S}^\times}
    \longrightarrow (\cO_{F,S}^\times/\mu_F)\otimes_{\Z}\R/\Z \longrightarrow 0.
\end{eqnarray}
In particular, it is a compact abelian group whose Pontryagin dual
$$
\Hom_{\cts}\left(\frac{\R^{S_{\f}}\!\times_{\R_{>0}}F_{\R}^\times}{\cO_{F,S}^\times},\R/\Z\right)
$$
is finitely generated, therefore the same is true of the closed subgroup $T_{F,S}$.

  Taking the Pontryagin dual of the right column of the commutative diagram above,
  we deduce that there is an equality
  \begin{eqnarray}\label{eq:Ar1} 
    [\Ara{F}]=[\Cl_{F,S}^\lor] + [\Hom_{\cts}(T_{F,S},\R/\Z)]
  \end{eqnarray}
  in $\G_{0}(\Z[G])$.

Combining \eqref{eq:Ar1} with the Pontryagin duals of \eqref{eq:S-torus} and \eqref{eq:Dirichlet} and
applying Lemma \ref{lem:tauv}, we see that there is an equality
\[
  [\Ara{F}] =  [\Cl_{F,S}^\lor] + d\cdot[\Z[G]] - [\Z^{S_\infty}] - [\mu_{F}^{\lor}] + [(\cO_{F,S}^\times/\mu_F)^*] - [(\Z^{S_{\f}})^*]
\]
in $\G_{0}(\Z[G])$.
We now deduce the result by applying Proposition~\ref{prop:dual} and noting that
there are $\Z[G]$-module isomorphisms $(\Z^{S_{\f}})^* \cong \Z^{S_{\f}}$ and
$\Z^S \cong \Z^{S_{\infty}} \oplus \Z^{S_{\f}}$.
\end{proof}\noindent

\begin{corollary}\label{cor:Ara}
  Let $F/K$ be a finite Galois extension of number fields, let $G$ be the Galois group,
  let $d$ be the degree of $K$ over $\Q$, and let $S$ be a finite $G$-stable set of places
  of $F$ containing all Archimedean places, and large enough for $\Cl_{F,S}$ to be trivial.
  Then the equality
  \[
    [\Ara{F}] = d\cdot [\Z[G]] - \involGrothen([\Z^S]-
    [\cO^\times_{F,S}])\label{eq:CorArTilde}
  \]
  holds in $\G_{0}(\Z[G])$.
\end{corollary}

\begin{proof}
The result follows by combining Proposition \ref{prop:Ara} with the observation
that $\involGrothen[\Z^S]=[\Z^S]$.
\end{proof}

\begin{corollary}\label{cor:indepS}
Let $F/K$ be a finite Galois extension of number fields, let $G$ be the Galois group, and
let $S$ and $S'$ be two finite $G$-stable sets of places of $F$, both containing all Archimedean places.
Then the equality
\[
  [\Z^{S'}] - [\cO_{F,S'}^\times] + [\Cl_{F,S'}] = [\Z^S] - [\cO_{F,S}^\times] + [\Cl_{F,S}]\\
\]
holds in $\G_{0}(\Z[G])$.
\end{corollary}

\begin{proof}
The result follows by combining Proposition \ref{prop:Ara} with the observation that
$\involGrothen[\Z^S] = [\Z^S]$ and $\involGrothen[\Z^{S'}] = [\Z^{S'}]$.
\end{proof}\noindent
As we will see in the next section, the $\Omega(3)$ conjecture, a standard conjecture
in the theory of Galois module structures, implies the following simpler expressions for
the classes of $\Ar_F$ and $\Ara{F}$.
\begin{conjecture}\label{conj:main}
Let $F/K$ be a finite Galois extension of number fields, let $G$ be the Galois group,
let $d$ be the degree of $K$ over $\Q$, let $S_{\infty}$ be the set of
Archimedean places of $F$, let $\mu_F$ be the group of roots of unity in $F$,
and let $\involGrothen$ be the automorphism of $\G_{0}(\Z[G])$ induced by the involution
$g\mapsto g^{-1}$ on $\Z[G]$, as defined in Section \ref{sec:duals}.
Then the following equalities hold in $\G_{0}(\Z[G])$:
\begin{enumerate}[leftmargin=*, label=\upshape{(\alph*)}]
  \item\label{item:conjmain1} $[\Ara{F}] = d\cdot [\Z[G]] - [\Z]$;
  \item\label{item:conjmain2} $[\Ar_F] = [\Z^{S_{\infty}}] - [\Z] - \involGrothen[\mu_F]$.
\end{enumerate}
\end{conjecture}\noindent
The next result shows that, in the two equations in Conjecture \ref{conj:main},
the difference between the left hand side and the right hand side is the same.
Hence each of \ref{item:conjmain1} and \ref{item:conjmain2} implies
the other, and is therefore equivalent to the entire conjecture.
\begin{lemma}\label{lem:equiv-assertions-mc-torsion}
  With the same notation as in Conjecture \ref{conj:main},
we have
\[
  [\Ara{F}]-[\Ar_F] = d \cdot [\Z[G]] - [\Z^{S_{\infty}}] + \involGrothen[\mu_F]
\]
in $\G_0(\Z[G])$.
Moreover, the common difference between the left hand side and the right hand side
of the two equations in Conjecture \ref{conj:main}
lies in the torsion subgroup $\G_0(\Z[G])_{\tors}$
of $\G_0(\Z[G])$.
\end{lemma}

\begin{proof}
  The first assertion follows from Proposition \ref{prop:Ara} with $S=S_{\infty}$.
  
We now prove the second assertion. By a theorem of Herbrand
there is an isomorphism
$(\Q\otimes_{\Z}\cO_F^\times) \oplus \Q \cong\Q^{S_{\infty}}$ of $\Q[G]$-modules,
see for example \cite[Chapter I, 4.3]{MR782485}.
Let $\theta: \G_0(\Z[G]) \rightarrow \G_0(\Q[G])$ be the map induced by 
the flat ring homomorphism $\Z[G] \rightarrow \Q[G]$.  
For every finitely generated $\Z[G]$-module $M$ we have $\theta(\involGrothen[M])=\theta([M])$.
Therefore applying $\theta$ to the first equality of Proposition \ref{prop:Ara}, we deduce
the equality
\[
  [\Q\otimes_{\Z} \Ar_F] - [\Q^{S_{\infty}}]+[\Q] = [\Q\otimes_{\Z} \cO_F^\times]- [\Q^{S_{\infty}}]+[\Q] = 0
\]
in $\G_0(\Q[G])$.
This shows that the difference between the left hand side and
the right hand side of Conjecture \ref{conj:main}\ref{item:conjmain2} lies in $\ker(\theta)$, 
which is equal to $\G_0(\Z[G])_{\tors}$ by \cite[Theorem (39.14)]{MR892316}.
The conclusion for Conjecture \ref{conj:main}\ref{item:conjmain1} follows from this and 
the first assertion.
\end{proof}
\noindent
For every ring $R$ that is flat over $\Z$, the functor
$R\otimes_{\Z}\bullet$ induces a homomorphism
$\G_0(\Z[G])\to \G_0(R[G])$. Hence, Conjecture \ref{conj:main} implies
the following conjecture, which will suffice for the applications to the
Cohen--Lenstra--Martinet heuristic in Section \ref{sec:CL}.

\begin{conjecture}\label{conj:main-for-CL}
Let $F/K$ be a finite Galois extension of number fields, let $G$ be the Galois group,
let $d$ be the degree of $K$ over $\Q$,
and let $P$ be a set of prime numbers not dividing $2 \cdot \# G$.
Then the equalities
\begin{align}\label{eq:ArAndArtilde}
  [\Z_{(P)}\otimes_{\Z}\Ara{F}] & = d\cdot [\Z_{(P)}[G]] - [\Z_{(P)}],\nonumber\\
  [\Z_{(P)}\otimes_{\Z}\Ar_F] & = [(\Z_{(P)})^{S_\infty}] - [\Z_{(P)}] -
    \involGrothen[\Z_{(P)}\otimes_{\Z}\mu_F]
\end{align}
hold in $\G_0(\Z_{(P)}[G])$.
\end{conjecture}
\begin{remark}\label{rmk:main-for-CL-in-abs-abelian-case}
Conjecture \ref{conj:main-for-CL} is equivalent to an affirmative answer to
\cite[Question 5.5]{MR4105790}. Thus in the special case that $G$ is abelian
and $K=\Q$, Conjecture \ref{conj:main-for-CL} was proven unconditionally in 
\cite[Theorem 5.4]{MR4105790}.  
\end{remark}
\begin{remark}\label{rmk:G0-tors-vanishes}
Let $F/K$ be a finite Galois extension of number fields and let $G$ be the
Galois group. Lemma \ref{lem:equiv-assertions-mc-torsion} shows that if
$\G_{0}(\Z[G])_{\tors}$ vanishes, then Conjecture \ref{conj:main} holds for
$F/K$. There are many cases in which $\G_0(\Z[G])_{\tors}$ has been calculated,
and among these are numerous instances in which it is trivial. In the case
that $G$ is abelian, a formula for $\G_{0}(\Z[G])$ is given in
\cite{MR601683}; this result has been extended to the case that $G$ is
nilpotent in \cite{MR953163}. Analogous observations also apply to Conjecture
\ref{conj:main-for-CL}.
\end{remark}\noindent

\section{The relation to Chinburg's $\Omega(3)$ conjecture}\label{sec:Chinburg}
\noindent
In this section we explain that Conjecture \ref{conj:main}, and a fortiori
Conjecture \ref{conj:main-for-CL}, follow from Conjecture \ref{conj:omega-3-over-max-order}
below, which is a weaker variant of Chinburg's $\Omega(3)$ conjecture,
Conjecture \ref{conj:omega-3} \cite[Conjecture 3.1]{MR786352}.

We first briefly review some material from algebraic $\K$-theory. 
We refer the reader to \cite[\S 38, \S 39, \S 49]{MR892316} for further details.

Let $R$ be a Dedekind domain and let $\Lambda$ be an $R$-order.
Let $\K_{0}(\Lambda)$ denote the Grothendieck group of the category
of finitely generated projective $\Lambda$-modules. By definition,
$\K_{0}(\Lambda)$ is the additive group generated by expressions $[P]$,
one for each isomorphism class of finitely generated projective
$\Lambda$-modules $P$, with relations $[P_{1} \oplus P_{2}] = [P_{1}] + [P_{2}]$ 
for all such modules $P_{1}, P_{2}$.
By \cite[Theorem (38.50)]{MR892316}, the group $\K_{0}(\Lambda)$ can be identified with
the Grothendieck group of the category of 
finitely generated $\Lambda$-modules of finite projective dimension.

For each maximal ideal $\fp$ of $R$, let $R_{\fp}$ denote the localisation of
$R$ at $\fp$, and define $\Lambda_{\fp} = R_{\fp} \otimes_{R} \Lambda$. 
A $\Lambda$-lattice $M$ is said to be \emph{locally free} if 
$\Lambda_{\fp} \otimes_{\Lambda} M$
is free over $\Lambda_{\fp}$ for every such $\fp$.
Note that every locally free $\Lambda$-lattice is projective by \cite[Proposition (8.19)]{MR632548}. 
Let $\K_{0}(\Lambda) \rightarrow \K_{0}(\Lambda_{\fp})$ be the map induced by
the ring homomorphism $\Lambda \rightarrow \Lambda_{\fp}$. Then the
\emph{locally free class group} $\Cl(\Lambda)$ is defined to be the kernel of
the homomorphism $\K_{0}(\Lambda) \rightarrow \prod_{\fp} \K_{0}(\Lambda_{\fp})$, 
where the product runs over all maximal ideals $\fp$ of $R$ (see \cite[Definition (39.12)]{MR892316}).
By \cite[(39.13)]{MR892316} we have
\begin{equation}\label{eq:locally-free-classgroup}
  \Cl(\Lambda) = \{ [\Lambda] - [L] \in \K_{0}(\Lambda) : L\textrm{ is a locally free }\Lambda\textrm{-lattice of rank }1\}. 
\end{equation}
Note that there are several equivalent definitions of $\Cl(\Lambda)$
(see \cite[\S 49A]{MR892316}, particularly \cite[p.\ 223]{MR892316}).

We now recall the statement of Chinburg's $\Omega(3)$ conjecture.
For the rest of the section,
let $F/K$ be a finite Galois extension of number fields and let $G$ be the Galois group.
For any finite $G$-stable set $S$ of places of $F$, let $X_{S}$ be the kernel of the
augmentation map $\Z^S \rightarrow \Z$ of $\Z[G]$-lattices.
Henceforth let $S$ be a finite $G$-stable set of places of $F$ such that 
\begin{enumerate}[leftmargin=*, label={(\roman*)}]
\item $S$ contains the set $S_{\infty}$ of Archimedean places of $F$, 
\item $S$ contains the ramified places of $F/K$, and
\item for every subfield $N$ of $F$ containing $K$, the ideal class group of
  $N$ is generated by the classes $\{[\fp\cap \cO_N] :\fp\in S\setminus S_{\infty}\}$.
\end{enumerate}
Tate \cite[p. 711]{MR0207680} defined a canonical class
$\alpha = \alpha_{S} \in \Ext_{\Z[G]}^{2}(X_{S}, \cO_{F,S}^{\times})$,
and showed the existence of so-called Tate sequences \cite[II, Th\'eor\`eme 5.1]{MR782485}, that is,
four term exact sequences of finitely generated $\Z[G]$-modules
\begin{equation}\label{eqn:tate-seq}
0 \longrightarrow \cO_{F,S}^{\times} \longrightarrow A \longrightarrow B \longrightarrow X_{S} \longrightarrow 0
\end{equation}
representing $\alpha$, where $A$ and $B$ are of finite projective dimension.
In \cite{MR786352}, Chinburg defined $\Omega(F/K,3) = [A]-[B] \in \K_{0}(\Z[G])$.
Moreover, he showed that $\Omega(F/K,3)$ lies in the locally free class group $\Cl(\Z[G])$, and
depends only on the extension $F/K$; in particular, it does not depend on the
choice of $S$ or on the choice of exact sequence \eqref{eqn:tate-seq}.

The root number class $W_{F/K} \in \Cl(\Z[G])$ was defined by Ph.\ Cassou-Nogu\`es
in the case that $F/K$ is at most tamely ramified,
and was generalised to wildly ramified extensions $F/K$ by Fr\"ohlich \cite{MR507603}.
It is an element of order at most $2$, and is defined in terms of the Artin root numbers of
the irreducible symplectic characters of $G$.
Moreover, if $G$ has no irreducible symplectic characters (for example, if
$G$ is abelian or of odd order), then $W_{F/K}$ is trivial by definition.

\begin{conjecture}[Chinburg's $\Omega(3)$ conjecture]\label{conj:omega-3}
The equality 
\[
\Omega(F/K,3)=W_{F/K}
\]
holds in $\Cl(\Z[G])$.
\end{conjecture}\noindent
Fix, for the rest of the section, a maximal order $\cM$ in $\Q[G]$ containing $\Z[G]$,
let $\rho\colon \Cl(\Z[G])\to \Cl(\cM)$ be the map
induced by the ring homomorphism $\Z[G] \rightarrow \cM$, and define
the \emph{kernel subgroup} $\D(\Z[G])$ of $\Cl(\Z[G])$ to be the kernel of $\rho$.
If $\cM'$ is any other maximal order in $\Q[G]$ containing $\Z[G]$,
and $\rho'\colon \Cl(\Z[G])\to \Cl(\cM')$ is the analogous map,
then by \cite[Theorem 1.6]{MR2905910} 
the kernel of $\rho$ is equal to that of $\rho'$.

\begin{conjecture}[Chinburg's $\Omega(3)$ conjecture modulo the kernel group]\label{conj:omega-3-over-max-order}
We have $\Omega(F/K,3) \equiv W_{F/K} \bmod \D(\Z[G])$. 
\end{conjecture}\noindent
The next result will be the main ingredient in the proof of Theorem \ref{thm:main} from the introduction.
\begin{proposition}\label{prop:omega-3-over-max-order-implies-conj:main}
Conjecture \ref{conj:omega-3-over-max-order} for $F/K$ implies Conjecture
\ref{conj:main} for $F/K$.
\end{proposition}\noindent
In order to prove this result, we will make use of the following lemma. 
Let
\[
\Psi(F/K) = [\cO_{F}^{\times}] - [X_{S_{\infty}}] - [\Cl_{F}] \in \G_{0}(\Z[G]).
\]

\begin{lemma}\label{lem:equiv-main-conj-and-main-conj-Chinburg-style}
Conjecture \ref{conj:main} for $F/K$ is equivalent
to the vanishing of $\Psi(F/K)$, while Conjecture \ref{conj:main-for-CL} is equivalent
to the assertion that for all sets $P$ of prime numbers not dividing $2\cdot \#G$
the image of $\Psi(F/K)$ under the map $\G_0(\Z[G])\to \G_0(\Z_{(P)}[G])$ is $0$.
Moreover, we always have $\Psi(F/K) \in \G_{0}(\Z[G])_{\tors}$.   
\end{lemma}

\begin{proof}
By Lemma \ref{lem:equiv-assertions-mc-torsion}, Conjectures
\ref{conj:main}\ref{item:conjmain1} and \ref{conj:main}\ref{item:conjmain2} are
equivalent to each other. By Corollary \ref{cor:Ara}, Conjecture \ref{conj:main}\ref{item:conjmain1}
is equivalent to the assertion that we have
\[
[\Z] = \involGrothen([\Z^S]-[\cO^\times_{F,S}])
\]
in $\G_{0}(\Z[G])$. Since $\involGrothen[\Z]=[\Z]$, this is equivalent to 
\[
[\cO^\times_{F,S}] = [\Z^S] - [\Z] = [X_{S}].
\]
By Corollary \ref{cor:indepS} with $S'=S_{\infty}$, this in turn is equivalent to $\Psi(F/K)=0$.

The proof of the claim regarding Conjecture \ref{conj:main-for-CL} is completely analogous.

The last claim follows from Lemma \ref{lem:equiv-assertions-mc-torsion}. 
\end{proof}\noindent
In light of Lemma \ref{lem:equiv-main-conj-and-main-conj-Chinburg-style}, Proposition
\ref{prop:omega-3-over-max-order-implies-conj:main} amounts to the statement that
Conjecture \ref{conj:omega-3-over-max-order} implies the vanishing of $\Psi(F/K)$.
This statement is well known to experts in Galois module theory;
see \cite[III]{MR724009}, \cite[\S 4, Proposition 6]{MR1110391} or \cite[\S 1]{MR1641555}.
We will recall the argument in the following proof and give some additional references.

\begin{proof}[Proof of Proposition \ref{prop:omega-3-over-max-order-implies-conj:main}.]
Let $\mu : \Cl(\Z[G]) \rightarrow \G_{0}(\Z[G])$ denote the restriction of the Cartan map
$\K_{0}(\Z[G]) \rightarrow \G_{0}(\Z[G])$,
which is induced by letting $\mu([M])=[M]$ if $M$ is a finitely generated projective $\Z[G]$-module.

Let $\xi \in \Cl(\Z[G])$.
Write $\xi = [\Z[G]] - [L]$, where $L$ is a locally free left ideal of $\Z[G]$,
which we may do by \eqref{eq:locally-free-classgroup}.
By \cite[Exercise 31.10]{MR632548} we have
\[
\xi = [\Z[G]] - [L] = [\cM] - [\cM \otimes_{\Z[G]} L] \textrm{ in } \G_{0}(\Z[G]).
\]
Hence we have the following commutative diagram of abelian groups:
\[
  \xymatrix{
   \Cl(\Z[G]) \ar[d]_{\rho}\ar[r]^{\mu} & \G_{0}(\Z[G]) \\
    \Cl(\cM) \ar[r]^{\mu'} & \G_{0}(\cM) \ar[u]^{\alpha},\\
  }
\]
where $\alpha$ is induced by restriction, and $\mu'$ is defined analogously to $\mu$.

Let $A$ and $B$ be as in the exact sequence \eqref{eqn:tate-seq}.
Then it follows from  the definition of $\Omega(F/K,3)$, from the exact sequence
\eqref{eqn:tate-seq}, and from Corollary \ref{cor:indepS} with $S'=S_{\infty}$, that
we have the equalities
\[
\mu(\Omega(F/K,3)) = [A] - [B] = [\cO_{F,S}^{\times}] - [X_{S}] = \Psi(F/K).
\]
Moreover, a special case of a result of Queyrut \cite[Proposition 2.3]{MR769765} shows that $\mu(W_{F/K})=0$.

Now suppose that Conjecture \ref{conj:omega-3-over-max-order} holds for $F/K$.
By definition of $\D(\Z[G])$, this is equivalent to $\rho(\Omega(F/K,3))=\rho(W_{F/K})$. 
Since $\mu$ factors via $\rho$, we have
\[
\Psi(F/K) = \mu(\Omega(F/K,3)) = \mu(W_{F/K}) = 0.
\]
The result now follows from Lemma \ref{lem:equiv-main-conj-and-main-conj-Chinburg-style}.
\end{proof}\noindent
We close the section by proving Theorem \ref{thm:main} from the introduction.
In fact, we prove the following stronger statement.
\begin{theorem}\label{thm:mainStronger}
  Let $F/K$ be a Galois extension of number fields, let $G$ be the Galois group,
  let $S_{\infty}$ be the $G$-set of Archimedean places of $F$, and let
  $\Lambda$ be the quotient of the group ring $\Z[\tfrac{1}{2\cdot \#G}][G]$
  by the two-sided ideal generated by $\sum_{g\in G}g$. Suppose that
  Conjecture \ref{conj:omega-3-over-max-order} holds for $F/K$, and that for every
  prime number $p$ not dividing $2\cdot\#G$, each primitive $p$-th root of unity
  in $F$ is in $K$. Then the equality
  \[
    [\Lambda\otimes_{\Z[G]}\Ar_F] - [\Lambda\otimes_{\Z[G]} \Z^{S_\infty}] = 0
  \]
  holds in $\G_{0}(\Lambda)_{\tors}$.
\end{theorem}
\begin{proof}
  Note that the fact that $[\Lambda\otimes_{\Z[G]}\Ar_F] - [\Lambda\otimes_{\Z[G]} \Z^{S_\infty}]$
  indeed lies in $\G_{0}(\Lambda)_{\tors}$ follows from Lemma \ref{lem:equiv-assertions-mc-torsion}.

  The ring homomorphism $\Z[G]\to \Lambda$ is a composition of two homomorphisms
  of the types discussed at the end of Section \ref{subsec:Groth-group-rings},
  so that it induces a group homomorphism $\G_0(\Z[G])\to \G_0(\Lambda)$.
  By Proposition \ref{prop:omega-3-over-max-order-implies-conj:main}, the hypotheses
  of the theorem imply that we have the equality
  \[
    [\Lambda\otimes_{\Z[G]} \Ar_F] = [\Lambda\otimes_{\Z[G]}\Z^{S_\infty}]-[\Lambda\otimes_{\Z[G]}\Z]-\involGrothen[\Lambda\otimes_{\Z[G]}\mu_F]
  \]
  in $\G_0(\Lambda)$, where $\involGrothen$ is the automorphism of $\G_0(\Lambda)$
  induced by the involution $g\mapsto g^{-1}$ on $G$. Write temporarily $\Z'=\Z[\tfrac{1}{2\cdot \#G}]$.
  Since the element $\tfrac{1}{\#G}\sum_{g\in G}g$ of $\Z'[G]$ acts
  trivially on the $\Z'[G]$-module $\Z'$, but is $0$ in the quotient $\Lambda$,
  we have $[\Lambda\otimes_{\Z[G]}\Z]=0$.

  Next, the hypothesis on roots of unity implies that for all prime numbers
  $p$ not dividing $2\cdot \#G$ and all $k\in \Z_{\geq 0}$, every $p^k$-th
  root of unity in $F$ is in $K$. This implies that $G$ acts trivially on
  $\Z'\otimes_{\Z}\mu_F$, hence, by the same argument as above, we have
  $[\Lambda\otimes_{\Z[G]}\mu_F]=0$. This completes the proof.
\end{proof}

\section{Some known cases of the conjectures}\label{sec:known}\noindent
In the present section we collect some situations in which Conjecture
\ref{conj:main-for-CL} is known. The main result of the section is Theorem
\ref{thm:main5}.

\begin{definition}\label{def:C-chi}
For a complex irreducible character $\chi$ of a finite group $G$, 
let $\Q(\chi)$ denote the field generated by the values of $\chi$, and
let $C(\chi)$ be the narrow class group of $\Q(\chi)$ if $\chi$ is symplectic,
and the usual ideal class group of $\Q(\chi)$ otherwise. 
\end{definition}

\begin{theorem}\label{thm:main5}
Let $F/K$ be a finite Galois extension of number fields, let $G$ be the Galois
group, let $G'$ be its commutator subgroup, and let $P$ be a set of prime
numbers not dividing $2\cdot\#G$. 
Suppose that $F^{G'}\!\!/\Q$ is abelian, and that for every complex irreducible character
$\chi$ of $G$ satisfying $\chi(1)>1$, the group $C(\chi)$ is generated by the
classes of the non-zero prime ideals of $\cO_{\Q(\chi)}$ not dividing any element
of $P$. Then Conjecture \ref{conj:main-for-CL} holds for $F/K$ and $P$.
\end{theorem}\noindent
The rest of the section is devoted to the proof of Theorem \ref{thm:main5} and some consequences.
\begin{theorem}\label{thm:omega3-in-abs-abelian-setting}
Let $F/K$ be a finite Galois extension of number fields and let $G$ be the Galois group. 
Suppose that $F/\Q$ is abelian. Then we have $\Omega(F/K,3)=W_{F/K}=0$ in $\Cl(\Z[G])$.
\end{theorem}

\begin{proof}
This is a special case of \cite[Corollary 1.4]{MR2290586}.
\end{proof}

\begin{corollary}\label{cor:abelian}
  Let $F/K$ be a finite Galois extension of number fields such that $F/\Q$
  is abelian. Then Conjecture \ref{conj:main} holds for $F/K$.
\end{corollary}

\begin{proof}
The assertion immediately follows from
Theorem \ref{thm:omega3-in-abs-abelian-setting} and Proposition
\ref{prop:omega-3-over-max-order-implies-conj:main}.
\end{proof}\noindent
If $G$ is a finite group, we denote by $\Irr_{\na}(G)$ the set of complex
irreducible characters of $G$ of degree greater than $1$, and
for $\chi$, $\chi'\in \Irr_{\na}(G)$ we
write $\chi\sim\chi'$ if $\chi$ and $\chi'$ are in the same Galois orbit, i.e.,
if there exists $\tau\in \Gal(\bar{\Q}/\Q)$ such that $\chi=\tau\circ\chi'$.

\begin{lemma}\label{lem:decomp-ZP[G]}
Let $G$ be a finite group and let $G'$ be its commutator subgroup. 
Then there is a direct product decomposition of $\Q$-algebras
\begin{equation}\label{eq:Q[G]-decomp}
\Q[G]\cong \Q[G/G'] \times \prod_{\chi\in \Irr_{\na}(G)/\sim}A_\chi, 
\end{equation}
where the product is taken over a full set of representatives of Galois
orbits of non-abelian characters of $G$, and each $A_\chi$ is a simple
$\Q$-algebra. 
Moreover, for a set $P$ of prime numbers not dividing $\# G$, there
is a corresponding direct product decomposition
\begin{equation}\label{eq:ZP[G]-decomp}
\Z_{(P)}[G] \cong \Z_{(P)}[G/G'] \times \prod_{\chi\in \Irr_{\na}(G)/\sim}R_\chi, 
\end{equation}
where each $R_{\chi}$ is a maximal $\Z_{(P)}$-order in $A_{\chi}$.
\end{lemma}

\begin{proof}
The decomposition \eqref{eq:Q[G]-decomp} is standard. 
By \cite[Proposition (27.1)]{MR632548}, the $\Z_{(P)}$-order $\Z_{(P)}[G]$ is
maximal in $\Q[G]$, so the decomposition \eqref{eq:ZP[G]-decomp}
follows from \cite[Theorem (10.5)]{MR0393100}.
\end{proof}

\begin{lemma}\label{lem:G0-tors-for-projections-onto-simple-components}
Let $G$ be a finite group, let $P$ be a set of prime numbers, let $\chi$ be a
complex irreducible character of $G$, let $A_{\chi}$ be the simple quotient
of $\Q[G]$ corresponding to the Galois orbit of $\chi$, let $R_{\chi}$ be the
image of $\Z_{(P)}[G]$ in $A_{\chi}$, and let $\Q(\chi)$ and $C(\chi)$ be as
in Definition \ref{def:C-chi}.
Then $\G_{0}(R_{\chi})_{\tors}$ is isomorphic to the quotient of $C(\chi)$ 
by the subgroup generated by the non-zero prime ideals of $\cO_{\Q(\chi)}$
not dividing any element of $P$.
\end{lemma}

\begin{proof}
This follows immediately from \cite[Theorem (38.67)]{MR892316}. 
\end{proof}
\begin{proof}[Proof of Theorem \ref{thm:main5}]
For an element $m$ of $\G_0(\Z[G])$, let $\Z_{(P)}\otimes m$ denote the
image of $m$ in $\G_0(\Z_{(P)}[G])$. Then Lemma
\ref{lem:equiv-main-conj-and-main-conj-Chinburg-style} implies that we have
$\Z_{(P)}\otimes \Psi(F/K)\in \G_0(\Z_{(P)}[G])_{\tors}$, and that
the claim is equivalent to the assertion that $\Z_{(P)}\otimes \Psi(F/K)=0$.
Moreover, by Lemma \ref{lem:decomp-ZP[G]}, there is a direct sum decomposition
of abelian groups
\[
\G_0(\Z_{(P)}[G])_{\tors}\cong\G_0(\Z_{(P)}[G/G'])_{\tors} \oplus \bigoplus_{\chi\in \Irr_{\na}(G)/\sim}\G_0(R_\chi)_{\tors}.
\]
Therefore the assertion that $\Z_{(P)}\otimes \Psi(F/K)=0$ is in turn equivalent
to the claim that the image of $\Psi(F/K)$ in each of the above summands is $0$.

Since $P$ does not contain any prime divisors of $\#G'$, it is
easily seen that the image of $\Psi(F/K)$ in $\G_0(\Z_{(P)}[G/G'])_{\tors}$
is equal to the image of $\Psi(F^{G'}/K)$. Since $F^{G'}/\Q$ is abelian, that
image is $0$ by Corollary \ref{cor:abelian}.

By Lemma \ref{lem:G0-tors-for-projections-onto-simple-components}, the hypotheses
of the theorem imply that $\G_0(R_\chi)_{\tors}$ is trivial for every
$\chi\in \Irr_{\na}(G)$, so the result follows.
\end{proof}\noindent
We conclude the section with a summary of what we know about Conjecture
\ref{conj:main-for-CL} for small groups.
\begin{proposition}\label{prop:112}
Let $F/\Q$ be a Galois extension of degree less than $112$, let $G$ be the Galois
group, and let $P$ be a set of prime numbers not dividing $2\cdot \#G$. Then 
Conjecture \ref{conj:main-for-CL} holds for $F/K$ and $P$, with $K=\Q$.
\end{proposition}
\begin{proof}
A direct computation, e.g. using the computational algebra
system \textsc{Magma} \cite{Wieb1997}, shows that the assumption that
$\#G<112$ implies that the hypotheses of Theorem \ref{thm:main5},
with $K=\Q$ and $P$ equal to the set of prime numbers not dividing $2\cdot \#G$,
are satisfied for $F/K$. A fortiori, the assumption is also satisfied
with $P$ being any smaller set.
\end{proof}

\begin{remark}
There are exactly two groups of order $112$ that have an irreducible
character $\chi$ of degree greater than $1$ such that, in the notation of the
proof of Theorem \ref{thm:main5}, the group
$\G_0(R_{\chi})_{\tors}$ is non-trivial. Each has exactly one Galois orbit of
such characters, in both cases of degree $2$. The two groups are both semidirect
products of a normal cyclic subgroup $C$ of order $56$ and a group
$H$ of order $2$. Let $x\in C$ be an element of order $7$, and let $y\in C$
be an element of order $8$. In one semidirect product, the non-trivial element
of $H$ acts on $C$ by $x\mapsto x^{-1}$ and $y\mapsto y^5$; and in the other
it acts on $C$ by $x\mapsto x^{-1}$ and $y\mapsto y^3$.
These are the two smallest Galois groups $G$ for which we do not currently know
Conjecture \ref{conj:main-for-CL} with $K=\Q$ and with $P$ being the set
of all prime numbers not dividing $2\cdot \#G$.
\end{remark}

\section{Counterexamples to the Cohen--Lenstra--Martinet heuristic}\label{sec:disproof}\noindent
In this section we prove Theorem \ref{thm:introdisproof} from the
introduction and explain how it disproves Heuristic \ref{he:orig}.
The following notation will remain in force throughout the section.
\begin{notation}\label{not:wreath}
  Let $p$ be an odd prime number, let $C_2$
  and $C_p$ be cyclic groups of orders $2$ and $p$, respectively, and let
  $G=C_2\wr C_p=C_2^p\rtimes C_p$, where $C_p$ acts on $C_2^p$ via its
  regular permutation action. 
  Let $\centre$ be the centre of $G$, which is
  cyclic of order $2$, and let $\gamma\in Z$ be the unique non-trivial
  central element. Let $P$ be the set of all prime numbers not dividing $2p$, and
  let $\Lambda$ be the quotient of $\Z_{(P)}[G]$ by the ideal generated by $1+\gamma$.
  Let $\cF$ be the family of all pairs $(F,i)$, where $F$ is a Galois number
  field satisfying $\mu_F=\{\pm 1\}$, and $i$ is an isomorphism from $G$ to the
  Galois group of $F$ sending $\centre$ to the inertia group of every Archimedean place.
\end{notation}

\begin{proposition}\label{prop:C2wreathCp}
All complex irreducible characters $\chi$ of $G$ with $\chi(1)>1$ satisfy $\Q(\chi)=\Q$.
\end{proposition}

\begin{proof}
Let $\chi$ be a complex irreducible character of $G$ with $\chi(1)>1$.
Then by \cite[\S 8.2]{Serre} we have 
$\chi=\Ind_{C_2^p}^G\psi$ for some irreducible character $\psi$ of $C_2^p$. 
Every such character $\psi$
satisfies $\Q(\psi)=\Q$, so we also have $\Q(\chi)=\Q$.
\end{proof}

\begin{corollary}\label{cor:C2wreathCp}
  Let $K=\Q$, and let $F$ be a Galois extension of $\Q$ with Galois group isomorphic to $G$.
Then Conjecture \ref{conj:main-for-CL} holds for $F/K$ and $P$. 
\end{corollary}

\begin{proof}
This follows from Proposition \ref{prop:C2wreathCp} and Theorem \ref{thm:main5}. 
\end{proof}

\begin{lemma}\label{lem:CMC2wreathCp}
There exist, up to isomorphism, infinitely many Galois number fields $F$ with Galois group isomorphic
to $G$ such that the inertia groups at infinity map to $\centre$ and
such that $\mu_{F}=\{\pm 1\}$.
\end{lemma}

\begin{proof}
Let $L/\Q$ be a cyclic extension of degree $p$ and let $H$ be the Galois group.
Since $[L:\Q]$ is odd and $L/\Q$ is Galois, $L$ must be totally real.
Let $l$ be a prime number that splits completely in $L$, and let $\fl$ be a
place of $L$ above $l$. By weak approximation,
there exists $a\in L^\times$ such that $a$ is totally negative, has
$\fl$-adic normalised valuation $1$, and for all $\sigma\in H\setminus\{1\}$
has $\sigma(\fl)$-adic valuation $0$. 
In particular, $a$ has the property that
for every non-empty subset $\Sigma\subseteq H$
the product $\prod_{\sigma\in \Sigma} \sigma(a)$ is not a square in $L^{\times}$.
The Galois closure $F=L(\{\sqrt{\sigma(a)} : \sigma\in H \})$
of $L(\sqrt{a})$ then has Galois group isomorphic to $G$ such that the inertia groups
at infinity map to $\centre$. 
Moreover, the maximal abelian extension inside $F$ is $F^{\ab}=
L(\sqrt{\Norm(a)})$ where $\Norm(a)=\prod_{\sigma\in H}  \sigma(a) \in \Q^{\times}$,
which has $l$-adic valuation $1$.
Thus $F^{\ab}/\Q$ is ramified at $l$, so as $l$ varies, we obtain
infinitely many extensions $F$. Of these, only finitely many can contain
a non-trivial cyclotomic field, which completes the proof.
\end{proof}\noindent
\begin{proposition}\label{prop:CM-extns}
For all $(F,i)\in \cF$ the class of $\Lambda \otimes_{\Z[G]}\Cl_F$ in $\G_{0}(\Lambda)$ is trivial.
\end{proposition}

\begin{proof}
Let $(F,i)\in \cF$. Under the hypotheses the extension $F/F^{\centre}$ is a
totally imaginary quadratic extension of a totally real field, and $\mu_{F} = \{ \pm 1 \}$,
which imply that $\Lambda\otimes_{\Z[G]} \cO_F^\times$ is trivial. Now Corollary
\ref{cor:C2wreathCp} implies that Conjecture \ref{conj:main-for-CL} holds in for
$F/K$ and $P$, and thus one has
$[\Z_{(P)}\otimes_{\Z}\Ar_F]  = [(\Z_{(P)})^{S_\infty}] - [\Z_{(P)}]$ in
$\G_{0}(\Z_{(P)}[G])$. Since $\gamma$ acts trivially on the two terms on the right,
this implies that $[\Lambda \otimes_{\Z[G]} \Ar_{F}]$ is trivial in $\G_{0}(\Lambda)$.
By the first equality of Proposition \ref{prop:Ara} we have
\[
[\Z_{(P)} \otimes_{\Z} \Ar_F]
=\involGrothen[\Z_{(P)} \otimes_{\Z} \cO_{F}^{\times}]-\involGrothen[\Z_{(P)} \otimes_{\Z} \Cl_{F}]
\]
in $\G_{0}(\Z_{(P)}[G])$. 
Since $\involGrothen$ is the automorphism of $\G_{0}(\Z_{(P)}[G])$ induced by the involution
$\sigma\mapsto \sigma^{-1}$ on $\Z[G]$, and since we have $\gamma^{-1}=\gamma$, 
we conclude that 
\[
  [\Lambda \otimes_{\Z[G]} \Cl_{F}] = [\Lambda \otimes_{\Z[G]} \cO_{F}^{\times}] = 0  
\]
in $\G_{0}(\Lambda)$, as desired.
\end{proof}\noindent
\begin{lemma}\label{lem:family-of-non-trivial-class-groups}
Let $p > 19$ be a prime number satisfying $p\equiv \pm3 \bmod 8$ and let $\zeta_{p}$ 
denote a primitive $p$-th root of unity. 
Then $\Cl(\Z[\zeta_{p}, \frac{1}{2p}])$ is non-trivial. 
\end{lemma}

\begin{proof}
Let $C$ be the class group of $\Q(\zeta_p)$ and let $C^{+}$ be the class group of the maximal
real subfield $\Q(\zeta_p)^{+}$. By \cite[Theorem 4.14]{MR1421575}, the natural
map $C^{+} \rightarrow C$ is an injection, and so we can and do view $C^{+}$ as
a subgroup of $C$. By \cite[Theorem 2.13]{MR1421575}, $2$ splits into $(p-1)/f$
distinct primes in $\Q(\zeta_{p})$, where $f$ is the multiplicative order of $2 \bmod p$.
The condition $p\equiv \pm3 \bmod 8$ is equivalent to $2$ being a quadratic non-residue
$\bmod$ $p$, and thus $f$ must be even. Since $\Q(\zeta_p)/\Q$ is cyclic, we conclude
that any prime of $\Q(\zeta_p)^+$ lying above 2 must be inert in $\Q(\zeta_p)/\Q(\zeta_p)^+$. 
Thus the class in $C$ of any prime of $\Q(\zeta_p)$ above $2$ must in fact lie in $C^+$.
Since $p>19$, we have that $\#C/\#(C^{+}) > 1$ by \cite[Corollary 11.18]{MR1421575}. 
Moreover, the unique prime of $\Q(\zeta_{p})$ above $p$ is principal (generated by $1-\zeta_{p})$.
Therefore the quotient of $C$ by the subgroup generated by the classes of primes above $2$ and $p$ 
is non-trivial, which gives the desired result.
\end{proof}

\begin{proposition}\label{prop:non-trivial-G0-tors-of-Lambda'}
Let $p>19$ be a prime number satisfying $p\equiv \pm3 \bmod 8$.
Then $\G_{0}(\Lambda)_{\tors}$ is non-trivial.
\end{proposition}

\begin{proof}
Since $G/G'$ is cyclic of order $2p$, we have that 
\[
\textstyle{\Z_{(P)}[G/G'] \cong \Z[\frac{1}{2p}] \times \Z[\zeta_{p}, \frac{1}{2p}] \times \Z[\frac{1}{2p}] \times \Z[\zeta_{p}, \frac{1}{2p}]}.
\]
Let $\Lambda'$ be the image of $\Lambda$ under the projection $\Z_{(P)}[G] \rightarrow \Z_{(P)}[G/G']$.
Then $\Lambda'$ is a direct factor of $\Lambda$ and $\Lambda' \cong \Z[\frac{1}{2p}] \times \Z[\zeta_{p}, \frac{1}{2p}]$.
Thus $\G_{0}(\Lambda')_{\tors}$ is a direct summand of $\G_{0}(\Lambda)_{\tors}$ and 
$\G_{0}(\Lambda')_{\tors} \cong \Cl(\Z[\zeta_{p}, \frac{1}{2p}])$ by 
\cite[Theorem (38.67)]{MR892316}.
Therefore the result follows from Lemma \ref{lem:family-of-non-trivial-class-groups}.
\end{proof}\noindent
\begin{remark}
  The conclusions of Lemma \ref{lem:family-of-non-trivial-class-groups} and
  Proposition \ref{prop:non-trivial-G0-tors-of-Lambda'} hold
  under the weaker assumption on $p$ that $p>19$ be a prime number such that
  $2$ generates a group of even cardinality in $(\Z/p\Z)^\times$, and the proofs
  carry over almost verbatim. The set of all such primes has been investigated,
  and Hasse \cite{Hasse} has computed its Dirichlet density to be $17/24$.
\end{remark}
\begin{proof}[Proof of Theorem \ref{thm:introdisproof}]
  Let $p$ be a prime number satisfying the hypothesis of Proposition
  \ref{prop:non-trivial-G0-tors-of-Lambda'}, and let $G$ and $\Lambda$ be as
  in Notation \ref{not:wreath}. Then $\G_0(\Lambda)_{\tors}$ is non-trivial. 
  The family $\cF$ as in Notation \ref{not:wreath} is infinite by Lemma \ref{lem:CMC2wreathCp}, and
  for all $(F,i)\in \cF$ the class of $\Lambda\otimes_{\Z[G]}\Cl_F$ in
  $\G_0(\Lambda)$ is trivial by Proposition \ref{prop:CM-extns}.
\end{proof}\noindent
Finally, the next result shows that Heuristic \ref{he:orig} is false for
$\cF$ and $\Lambda$ as above and for some natural functions $f$.
\begin{proposition}\label{prop:CLM-equidistr}
Suppose that for
all homomorphisms $\phi\colon \G_{0}(\Lambda)_{\tors}\to \C^\times$,
Heuristic \ref{he:orig}
holds with $K=\Q$ and with $f$ being the function that assigns to a finite
$\Lambda$-module $M$ the value of $\phi$ on the class of $M$ in
$\G_{0}(\Lambda)_{\tors}$. Then as $(F,i)$ ranges over $\cF$, the class of
$\Lambda\otimes_{\Z[G]}\Cl_F$ in $\G_{0}(\Lambda)_{\tors}$ is equidistributed.
\end{proposition}

\begin{proof}
By Lemma \ref{lem:decomp-ZP[G]} there is a direct product decomposition
\[
  \Z_{(P)}[G] \cong \Z_{(P)}[G/G'] \times \prod_{\chi\in \Irr_{\na}(G)/\sim}R_\chi.
\]
Note that the group $G/G'$ is cyclic of order $2p$.
Let $\bar{\gamma}$ be the image of $\gamma$ under the projection map $G \rightarrow G/G'$,
and let $\bar{\Lambda}$ be the quotient of $\Z_{(P)}[G/G']$ by the two-sided
ideal generated by $1+\bar{\gamma}$. Then $\Lambda \cong \bar{\Lambda} \times T$,
where $T$ is a direct factor of $\prod_{\chi\in \Irr_{\na}(G)/\sim}R_\chi$.
Moreover, for $\chi\in \Irr_{\na}(G)$, we have that $\Q(\chi)=\Q$ by
Proposition \ref{prop:C2wreathCp}, and thus $\G_{0}(R_{\chi})_{\tors}$ is
trivial by Lemma \ref{lem:G0-tors-for-projections-onto-simple-components}.
Hence there is a canonical isomorphism
$\G_{0}(\Lambda)_{\tors} \cong \G_{0}(\bar{\Lambda})_{\tors}$.
The set $\cF$ is infinite by Lemma \ref{lem:CMC2wreathCp}.
Note that $F^{G'}$ is an imaginary abelian number field that is a quadratic
extension of a real field, and $\bar{\Lambda}\otimes_{\Z[G]}\Cl_F$ is the
maximal quotient of $\Z_{(P)}\otimes_{\Z}\Cl_{F^{G'}}$ on which complex
conjugation acts by $-1$. Equidistribution of minus parts of class groups
in the corresponding Grothendieck group was shown in \cite[Proposition 4.4]{MR4105790}
to follow from the Cohen--Lenstra--Martinet heuristic for families of
imaginary extensions \emph{with Galois group isomorphic to $G/G'$}.
However, it is easy to see that the prediction is the same in the
present situation, despite the fields being ordered differently; see Remark \ref{rmrk:expvaluequo}.
\end{proof}

\bibliography{OrientedArakelov}

\providecommand{\bysame}{\leavevmode\hbox to3em{\hrulefill}\thinspace}
\providecommand{\MR}{\relax\ifhmode\unskip\space\fi MR }
\providecommand{\MRhref}[2]{%
  \href{http://www.ams.org/mathscinet-getitem?mr=#1}{#2}
}
\providecommand{\href}[2]{#2}
\begin{thebibliography}{CNCFT91}

\bibitem[BCP97]{Wieb1997}
W.~Bosma, J.~Cannon, and C.~Playoust, \emph{The {M}agma algebra system. {I}.
  {T}he user language}, J. Symbolic Comput. \textbf{24} (1997), no.~3-4,
  235--265, Computational algebra and number theory (London, 1993).

\bibitem[BF06]{MR2290586}
D.~Burns and M.~Flach, \emph{On the equivariant {T}amagawa number conjecture
  for {T}ate motives. {II}}, Doc. Math. (2006), no.~Extra Vol., 133--163.
  \MR{2290586}

\bibitem[BLJ17]{MR3705226}
A.~Bartel and H.~W. Lenstra~Jr., \emph{Commensurability of automorphism
  groups}, Compos. Math. \textbf{153} (2017), no.~2, 323--346. \MR{3705226}

\bibitem[BLJ20]{MR4105790}
\bysame, \emph{On class groups of random number fields}, Proc. Lond. Math. Soc.
  (3) \textbf{121} (2020), no.~4, 927--953. \MR{4105790}

\bibitem[Chi83]{MR724009}
T.~Chinburg, \emph{On the {G}alois structure of algebraic integers and
  {$S$}-units}, Invent. Math. \textbf{74} (1983), no.~3, 321--349. \MR{724009}

\bibitem[Chi85]{MR786352}
\bysame, \emph{Exact sequences and {G}alois module structure}, Ann. of Math.
  (2) \textbf{121} (1985), no.~2, 351--376. \MR{786352}

\bibitem[CKPS98]{MR1641555}
T.~Chinburg, M.~Kolster, G.~Pappas, and V.~Snaith, \emph{Galois structure of
  {$K$}-groups of rings of integers}, $K$-Theory \textbf{14} (1998), no.~4,
  319--369. \MR{1641555}

\bibitem[CLJ84]{MR756082}
H.~Cohen and H.~W. Lenstra~Jr., \emph{Heuristics on class groups of number
  fields}, Number theory, {N}oordwijkerhout 1983 ({N}oordwijkerhout, 1983),
  Lecture Notes in Math., vol. 1068, Springer, Berlin, 1984, pp.~33--62.
  \MR{756082}

\bibitem[CM87]{CMII}
H.~Cohen and J.~Martinet, \emph{{C}lass {G}roups of {N}umber {F}ields:
  {N}umerical {H}euristics}, Math. Comp. \textbf{48} (1987), no.~177, 123--137.

\bibitem[CM90]{MR1037430}
\bysame, \emph{\'{E}tude heuristique des groupes de classes des corps de
  nombres}, J. Reine Angew. Math. \textbf{404} (1990), 39--76. \MR{1037430}

\bibitem[CNCFT91]{MR1110391}
Ph. Cassou-Nogu{\`e}s, T.~Chinburg, A.~Fr{\"o}hlich, and M.~J. Taylor,
  \emph{{$L$}-functions and {G}alois modules}, {$L$}-functions and arithmetic
  ({D}urham, 1989), London Math. Soc. Lecture Note Ser., vol. 153, Cambridge
  Univ. Press, Cambridge, 1991, Based on notes by D. Burns and N. P. Byott,
  pp.~75--139. \MR{1110391}

\bibitem[CR81]{MR632548}
C.~W. Curtis and I.~Reiner, \emph{Methods of representation theory. {V}ol.
  {I}}, John Wiley \& Sons, Inc., New York, 1981. \MR{632548}

\bibitem[CR87]{MR892316}
\bysame, \emph{Methods of representation theory. {V}ol. {II}}, Pure and Applied
  Mathematics (New York), John Wiley \& Sons, Inc., New York, 1987. \MR{892316}

\bibitem[Fr{\"{o}}78]{MR507603}
A.~Fr{\"{o}}hlich, \emph{Some problems of {G}alois module structure for wild
  extensions}, Proc. London Math. Soc. (3) \textbf{37} (1978), no.~2, 193--212.
  \MR{507603}

\bibitem[Has66]{Hasse}
H.~Hasse, \emph{\"{U}ber die {D}ichte der {P}rimzahlen {$p$}, f\"{u}r die eine
  vorgegebene ganzrationale {Z}ahl {$a\not=0$} von gerader bzw. ungerader
  {O}rdnung {${\rm mod}.p$} ist}, Math. Ann. \textbf{166} (1966), 19--23.
  \MR{205975}

\bibitem[HTW88]{MR953163}
I.~Hambleton, L.~R. Taylor, and E.~B. Williams, \emph{On {$G_n(RG)$} for {$G$}
  a finite nilpotent group}, J. Algebra \textbf{116} (1988), no.~2, 466--470.
  \MR{953163}

\bibitem[Lan94]{MR1282723}
S.~Lang, \emph{Algebraic number theory}, second ed., Graduate Texts in
  Mathematics, vol. 110, Springer-Verlag, New York, 1994. \MR{1282723}

\bibitem[Len81]{MR601683}
H.~W. Lenstra, Jr., \emph{Grothendieck groups of abelian group rings}, J. Pure
  Appl. Algebra \textbf{20} (1981), no.~2, 173--193. \MR{601683}

\bibitem[NSW08]{MR2392026}
J.~Neukirch, A.~Schmidt, and K.~Wingberg, \emph{Cohomology of number fields},
  second ed., Grundlehren der Mathematischen Wissenschaften [Fundamental
  Principles of Mathematical Sciences], vol. 323, Springer-Verlag, Berlin,
  2008. \MR{2392026}

\bibitem[Que85]{MR769765}
J.~Queyrut, \emph{Anneaux d'entiers dans le m\^{e}me genre}, Illinois J. Math.
  \textbf{29} (1985), no.~1, 157--179. \MR{769765}

\bibitem[Rei75]{MR0393100}
I.~Reiner, \emph{Maximal orders}, Academic Press [A subsidiary of Harcourt
  Brace Jovanovich, Publishers], London-New York, 1975, London Mathematical
  Society Monographs, No. 5. \MR{0393100}

\bibitem[Sch08]{Schoof}
R.~Schoof, \emph{Computing {A}rakelov class groups}, Algorithmic number theory:
  lattices, number fields, curves and cryptography, Math. Sci. Res. Inst.
  Publ., vol.~44, Cambridge Univ. Press, Cambridge, 2008, pp.~447--495.
  \MR{2467554}

\bibitem[Ser77]{Serre}
J.-P. Serre, \emph{Linear representations of finite groups}, Springer-Verlag,
  New York-Heidelberg, 1977, Translated from the second French edition by
  Leonard L. Scott, Graduate Texts in Mathematics, Vol. 42. \MR{0450380}

\bibitem[Tat66]{MR0207680}
J.~Tate, \emph{The cohomology groups of tori in finite {G}alois extensions of
  number fields}, Nagoya Math. J. \textbf{27} (1966), 709--719. \MR{0207680}

\bibitem[Tat84]{MR782485}
\bysame, \emph{Les conjectures de {S}tark sur les fonctions {$L$} d'{A}rtin en
  {$s=0$}}, Progress in Mathematics, vol.~47, Birkh\"{a}user Boston, Inc.,
  Boston, MA, 1984, Lecture notes edited by Dominique Bernardi and Norbert
  Schappacher. \MR{782485}

\bibitem[Ull77]{MR2905910}
S.~V. Ullom, \emph{A survey of class groups of integral group rings}, Algebraic
  number fields: {$L$}-functions and {G}alois properties, Academic Press,
  London, 1977, pp.~497--524. \MR{2905910}

\bibitem[Was97]{MR1421575}
L.~C. Washington, \emph{Introduction to cyclotomic fields}, second ed.,
  Graduate Texts in Mathematics, vol.~83, Springer-Verlag, New York, 1997.
  \MR{1421575}

\end{thebibliography}
\bibliographystyle{amsalpha}

\end{document}